\documentclass[11pt]{article}
\usepackage{amsmath}
\usepackage{amssymb}
\usepackage{amsthm}
\usepackage{amsfonts}
\usepackage{graphicx}
\usepackage{float}
\usepackage[scriptsize,nooneline]{subfigure}
\usepackage{indentfirst}
\usepackage{cite}
\usepackage{color}
\usepackage{epstopdf}
\usepackage{setspace}
\usepackage{booktabs}
\usepackage[labelfont=footnotesize,font=footnotesize]{caption}
\newtheorem{theorem}{Theorem}
\newtheorem{corollary}{Corollary}
\newtheorem{lemma}{Lemma}

\newtheorem{definition}{Definition}

\newtheorem{proposition}{Proposition}
\newtheorem{remark}{Remark}

\begin{document}
\title{\bf{A state-dependent vector control for a West Nile Virus model from mosquitoes to birds$^\dag$}
\footnotetext{$^\dag$This research has been partially supported by the National Natural Science Foundation of China (Grant No. 11461067 and 11271312).}
\footnotetext{$^\ddag$Corresponding author, Tel/Fax: +86-991-858-5505. E-mail: lfnie@163.com.}}

\author{Lin-Fei Nie$^{\ddag}$\quad Jing-Yun Shen\\
{\footnotesize \textit{College of Mathematics and
Systems Science, Xinjiang University, Urumqi {\small 830046}, P.R. China}}\\
{\footnotesize }}

\date{}
\maketitle

\begin{abstract}
In this paper, a novel West Nile Virus model looking upon the infected birds as monitoring threshold, for the mosquitoes and birds with impulsive state feedback control is considered. We obtain sufficient conditions of the global asymptotical stability of the system without impulsive state feedback control via comprehensively qualitative analysis. By using the Poincar\'e map, we obtain that the system with impulsive state feedback control has a positive periodic solution of order-1 or order-2 which is asymptotical stability due to the analogue of Poincar\'e criterion, theory of differential inequalities, differential equation geometry and so on. What's more, sufficient conditions for existence and stability of the order one periodic solution are given by the existence and uniqueness of the limit. Our results show that the control measure is effective and feasible by means of numerical simulations.

\vspace{2mm}
\end{abstract}

\section{Introduction}
West Nile Virus(WNV) is a mosquito-borne that belongs to the encephalitis virus group within the family Flaviviridae \cite{Gabriel}. It first appeared in the United States in 1999, reached  Pennsylvania (PA) in 2000, and has since spread across PA as well as the rest of the United States. WNV has developed of a national donor testing program so far. It infects various mammal species specially horses and humans; however, the most commonly WNV infected animals are birds that serves as the reservoir host, the main path of WNV transmission is by the bite of an infected mosquitoes whose the salivary fluid \cite{Meshkat} and \cite{Garcia}. The natural pathway of the virus is from bird to mosquito to bird; therefore WNV is primarily spread by migrating birds, dispersal of resident birds, mosquito movement, or human transportation of mosquitos, birds and animals \cite{Henninga,Kilpatrick}. Millions of birds have died from WNV and for some species and locales, more than 50 percent of the population has perished. At least seven species of birds, all of whom are residents in PA, have suffered serious
declines in population because of WNV \cite{Henninga}.

Much has been done in terms of modeling and analysis of transmission dynamics of WNV, see, for example, Bowman et al.\cite{Bowman} has proposed a mathematical modelling and analysis to assess two main anti-WNV preventive strategies, namely: mosquito reduction strategies and personal protection, similarly, \cite{Wan, Wonham}. Fan et al. have considered the effect of temperature on the transmission of the virus spread in \cite {Fan}. Taking effective measures to curb the rage of WNV is self-evident. The more ordinary strategies incorporate the way of culling mosquitoes seeing in \cite{Xu}, which was considered the strategy of implementing periodic culling of mosquitoes at critical times, namely, the control of WNV with fixed moments, such that the death rate of mosquito was enhanced. Popular tactics were adopted to control the reproduction rate of mosquito by the sterile insect technique  which has also been applied to reduce or eradicate the wild mosquitoes. For example, utilizing radical or other chemical or physical methods, male mosquitoes are genetically modified to be sterile so that they are incapable of producing offspring despite being sexually active in \cite{Cai}. Zhang et al. \cite{Zhang} put forward periodic patterns and pareto efficiency of state dependent impulsive controls regulating interactions between wild and transgenic mosquito populations looking on the total size of the wild and transgenic mosquito populations as threshold value. Certain sufficient conditions for the existence and orbital stability of positive order-1 periodic solution of the model with state-dependent impulsive perturbations were attained.

 A question arises about how much the density of infected birds should be cured in combination with pesticide release in order to effectively control the spread of WNV. To address this question, which is the aim of this paper, we formulate a model for the interaction between mosquitoes and birds based on a three-dimensional ODE system with state-dependent impulsive perturbations. The state-dependent impulsive model is proposed for intergrated pest management such that the pest of the population size is no larger than its economic threshold(ET) in \cite{Tangsanyi-1,Tangsanyi-2} and the existence and stability of the order-1 periodic solution are given. Nie et al. \cite{Nie-1,Nie-2} employed the strategies of state dependent pulse vaccination in epidemic model which is more suitable to actual circumstances. What's more, a microbial pesticide model with impulsive state feedback control was proposed in \cite{Wang-2} and so on.

Motivated by the above-mentioned considerations and by the ideas of \cite{Zhang, Xu}. We formulate a model for the interaction between mosquitoes and birds based on a three-dimensional ODE system which is subject to state-dependent impulsive peturations. This paper is structed as follows. In the next section, we formulate mathematical model\eqref{System-1} considering state-dependent pesticide sprays and curing the infected birds as control measures and take into account the generalized planar impulsive system. Section 3, the qualitative analysis of the system without impulsive effect is given. Section 4, we investigate the existence and stability of the order-$k(1,2)$ periodic solution in different cases. In order to asses the control strategy, sensitivity analysis is applied to study the contribution of each parameter on the disease transmission in Section 5. Finally, we give brief discussions on our findings.

\section{Model Formulation and Preliminaries}

A schematic description of the traditional WNV model from vector mosquitoes to birds can be depicted in Figure 1(a). Here, the total female mosquito population is split into the populations of susceptible ($S_m$) and infected ($I_m$) mosquitoes, and the total bird population ($N_b$) is divide into the susceptible ($S_b$) and infected ($I_b$) birds. However, in a real world application, it is very difficult to distinguish from susceptible and infected mosquitoes. Additionally, George et al. \cite{George} estimated the impact of WNV on the survival of avian populations and find that populations were negatively affected by WNV in 23 of the 49 species studied (47\%). Based on these, it is reasonable to simplify the transmission rules of WNV from the flow chart Figure 1(a) to Figure 1(b) for an analysis
of mathematical reasoning. In it, we classify female mosquitoes and birds into three subclass: female mosquitoes ($M$), susceptible birds ($S_b$) and infected birds ($I_b$).

\begin{figure}[htb]
\centering
\includegraphics[width=0.90\textwidth,height=0.33\textwidth]{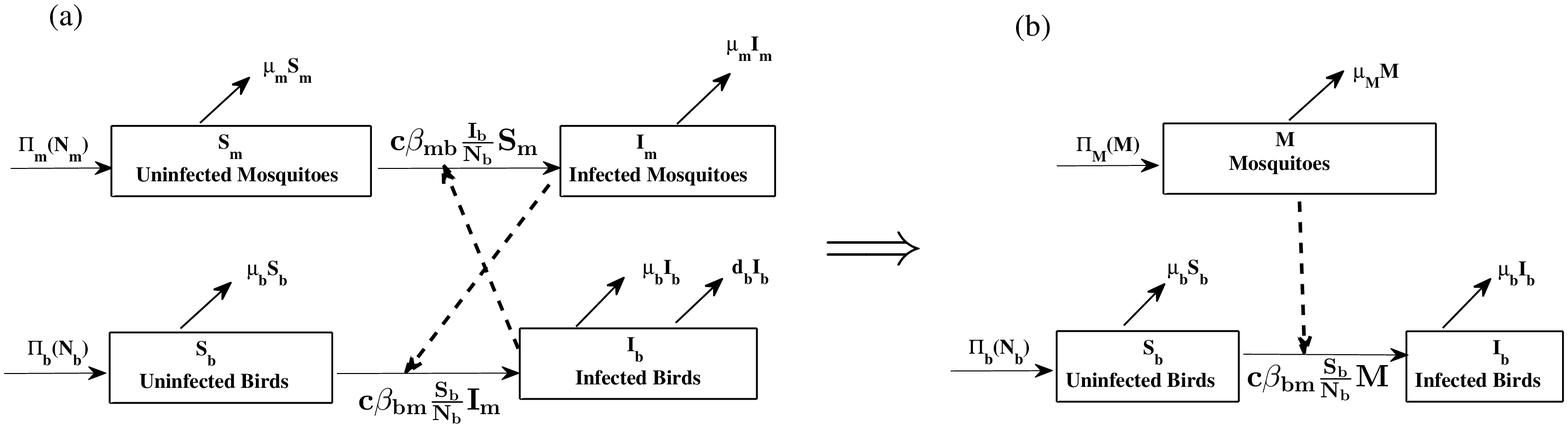}
\caption{Flow chart of the transmission of West Nile virus between vector mosquitoes and birds.}
\label{Fig.1}
\end{figure}

Further, in order to control the transmission of WNV, it's essential to reduce the quantities of infected birds and mosquitoes. Due to the consideration of environmental protection and  economic support, the effective and easily implemented control strategies are only applied when the quantity of the infected birds reaches a hazardous threshold value but not at any state. To achieve this objective, a host vector dynamics model with state dependent impulsive control strategies for the transmission of WNV reads
\begin{equation}\label{System-1}
\left\{\begin{aligned}&\left.
\begin{aligned}
& \frac{\mathrm{d}M(t)}{\mathrm{d}t}=\mu_m M(t)\left(1-\frac{M(t)}{K_m}\right)-\delta_m M(t)\\
&\frac{\mathrm{d}S_b(t)}{\mathrm{d}t}=\mu_{b}(S_b(t)+I_b(t))-c\beta_{bm}\frac{ S_b(t)}{N_b(t)}M(t)-\mu_{b}S_{b}(t)\\
&\frac{\mathrm{d}I_b(t)}{\mathrm{d}t}=c\beta_{bm}\frac{ S_b(t)}{N_b(t)}M(t)-\mu_{b}I_{b}(t)
\end{aligned}\right\}
 I_b(t)< H_b,\\
&\left.
\begin{aligned}
&M(t^+)=(1-p)M(t)\\
&S_b(t^+)=S_b(t)+qI_b(t)\\
&I_b(t^+)=(1-q)I_b(t)
\end{aligned}
\right\}
I_b(t)= H_b.
\end{aligned}
\right.
\end{equation}
The means of model \eqref{System-1} as follows: when the quantity of infected bird $I_b$ reaches the critical threshold value $H_b$ at time $t_i(H_b)$ at the $i$-th time, control measures (such as soluble powder drug blend in daily birds' drinking water or breeders can feed grain which has been soaked in liquid medicine to birds, spraying pesticides to reduce the quantity of mosquitoes) are taken, then the quantities of mosquitoes, susceptible and infected birds immediately become $(1-p)M(t_i(H_b))$, $(1-q)I_b(t_i(H_b))$ and$ S_b(t_i(H_b))+qI_b(t_i(H_b))$, respectively. The parameters and variables of the model are described in Table \ref{Table1}.

\begin{table}[htpb]
\centering
\caption{List of parameters for WNV transmission}
\vspace{3mm}
\begin{tabular}{llcc}
\toprule
Parameter & Definition & Range & Source \\
\midrule
$\mu_m$ & Mosquitoes per capita birth rate ($day^{-1}$) & $0.036\sim42.5$& \cite{Wonham}\\[1mm]
$K_m$ & Environmental carrying capacity of mosquitoes & $10^5\sim10^6$ & \cite{Wonham} \\[1mm]
$\delta_m$ & Natural death rate of mosquitoes ($day^{-1}$) & $0.016\sim0.07$ & \cite{Wonham}\\[1mm]
$\mu_b$ & Recruitment/Death rate of birds ($day^{-1}$) & $10^{-4}\sim10^{-3}$& \cite{Wonham}\\[1mm]

$c$  & Biting rate of mosquitoes & $0.09\sim0.16$ &  \cite{Wonham}\\[1mm]

$\beta_{bm}$ & Transmission probability from mosquitoes to birds & $0.80\sim0.96$ & \cite{Bowman,Wonham}\\[1mm]

$N_b$ & The total bird population & $-$ & Assumed \\[1mm]
$p$ & Culling rate of mosquitoes & $0\sim1$ & Assumed \\[1mm]
$q$ & Rate of infected birds cured well & $0\sim1$ & Assumed \\[1mm]
\bottomrule
\end{tabular}\label{Table1}
\end{table}

It follows from the second and third equations of model \eqref{System-1} that
$$
\mathrm{d}S_b(t)/\mathrm{d}t+\mathrm{d}I_b(t)/\mathrm{d}t=0,
$$
which means the total number of birds is a constant. We might as well set $S_b(t)+I_b(t)=N_b$. Model \eqref{System-1} is equivalent to the following simplified model
\begin{equation}\label{System-2}
\left\{\begin{aligned}&\left.
\begin{aligned}
& \frac{\mathrm{d}M(t)}{\mathrm{d}t}=\mu_m M(t)\left(1-\frac{M(t)}{K_m}\right)-\delta_m M(t)\\
&\frac{\mathrm{d}I_b(t)}{\mathrm{d}t}=c\beta_{bm}\left(1-\frac{ I_b(t)}{N_b}\right)M(t)-\mu_{b}I_{b}(t)
\end{aligned}\right\}
 I_b(t)< H_b,\\
&\left.
\begin{aligned}
&M(t^+)=(1-p)M(t)\\
&I_b(t^+)=(1-q)I_b(t)
\end{aligned}
\right\}
I_b(t)= H_b.
\end{aligned}
\right.
\end{equation}

The global existence and uniqueness of solution for model
\eqref{System-2} is guaranteed by the smoothness of the
right-hand sides of model \eqref{System-2}. For more
details, we refer to \cite{VLakshmikantham}. To discuss model \eqref{System-2}, the generalized planar impulsive semi-dynamical systems with state-dependent feedback control can be firstly introduced as follows:

Letting $\Omega =I(\Gamma)$ is the phase set (i.e. for any $ Z\in \Gamma$, $I(Z)=Z^+\in \Omega $), and $\Omega \bigcap \Gamma=\emptyset$. Obviously, system (\ref{System-2}) is generally known as a planar impulsive semi-dynamical system, where impulsive set $\Gamma=\{ (M, I_b)\in R^2: M > 0, I_b=H_b\}$ is a subset of $R^2$  and continuous function
$$
I: (M, H_b)\in \Gamma \rightarrow (M^+, I_b^+)=((1-p)M,(1-q)H_b)\in\Omega.
$$
It follows that the phase set is
$$
\Omega=I(\Gamma)=\{(M^+, I_b^+)\in R^2 : M^+ =(1-p)M,  I_b^+ =(1-q)H_b\}.
$$

For any $Z\in R^2$, we define the positive orbit as
follows
$$
\Pi^+(Z_0,t_0)=\{Z(t)=(M(t),I_b(t)): t\geq
t_0,Z(t_0)=Z_0\}.
$$

We mainly discuss the existence of periodic solution of model \eqref{System-2} by the existence criteria of the general
impulsive autonomous system. Before introducing the existence criteria, we introduce and quote some topological definitions  about impulsive differential equations as following, more details to see \cite{Li}.

\begin{definition}\label{definition 2}
A trajectory $\Pi^+(Z)$ of $(R^2, \Pi; \Gamma, I)$  is said to be order $k$ periodic if there exist nonnegative
integers $m$ and $k$ such that for which $I^m(Z)=I^{m+k}(Z)$ and $Z\in \Gamma$.
\end{definition}

\begin{definition}\label{definition 3}
For any two points $A_i(N_{mi}, H_b), A_j(N_{mj}, H_b) \in \Gamma$ experiencing after one time impulsive effect are ${A}_i^+(N_{mi}^+, (1-q)H_b), {A}_j^+(N_{mj}^+, (1-q)H_b) \in \Omega$, respectively, where $N_{mi}^+ =(1-p)N_{mi}, N_{mj}^+=(1-p)N_{mj}$, $i,j=0,1,2\cdots$ then we called $\widetilde{A}_i^+$ =$\widetilde{A}_j^+$ if and only if $N_{mi}^+=N_{mj}^+$; $\widetilde{A}_i^+ < \widetilde{A}_j^+$ if and only if $N_{mi}^+ < N_{mj}^+$; $\widetilde{A}_i^+ > \widetilde{A}_j^+$ if and only if $N_{mi}^+ > N_{mj}^+$.
\end{definition}
In order to investigate the existence of order-$k$ periodic solution. The Poincar$\acute{e}$ map is needed to be constructed firstly.
Assuming that any trajectory $\Pi({Z_0},t_0)$ with initial value ${Z_0} =(N_{m0}, I_{b0})\in \Omega$ experiences impulsives $k+1$ times(finite or infinite), then the corresponding coordinates are denoted as $M_i=(N_{mi}, H_b)\in\Gamma$ and $M_i^+=(N_{mi}^+, (1-q)H_b)\in\Omega (i=1,2,3,\cdots,k)$, where the point $M_i^+$ is the impulsive point of $M_i$ after one time impulsive effect. Therefore, if both points $M_i^+$ and $M_{i+1}$ lie in the same trajectory $\Pi({z_0},t_0)$ for $i=1,2,\cdots,k$, then the point $M_{i+1}^+$ is only determined by $M_i^+$, which can be expressed by $N_{m(i+1)}=f(N_{mi}^+)$. Clearly, function $f$ is continuously differentiable accordingly to the Cauchy-Lipchitz theorem. Therefore, the Poincar\'e map $\digamma$ can be defined as
\begin{equation}\label{equation-1}
N_{m(i+1)}^+=(1-p)f(N_{mi}^+)\triangleq\digamma(N_{mi}^+),i=1,2\cdots k.
\end{equation}

From \eqref{equation-1} we could know that impulsive point of model \eqref{System-2} defined in phase set can be described as
$N_{m(i+1)}^+=\digamma(N_{mi}^+)$. Since function $f$ is continuously differentiable with respect to $N_{mi}^+$ under the $\mu_m >\delta_m$, consequently Poincar\'e map $\digamma$ is continuously differentiable with respect to $N_{mi}^+$.
From a biological
point of view, we only consider model \eqref{System-2} in the region $\mathbb{R}_+^2=\{(M,I_b): M\geq 0, I_b\geq 0\}$. Obviously, $\mathbb{R}_+^2$ is divided into four domains with
vertical isocline $\mathrm{d}M(t)/\mathrm{d}t=0$ and  horizontal isocline
$\mathrm{d}I_b(t)/\mathrm{d}t=0$, followed by
\begin{equation}\label{Domain}
\begin{aligned}
&\Omega_1:=\left\{(M,I_b)\in\mathbb{R}_+^2: \frac{\mathrm{d}M(t)}{\mathrm{d}t}>0,
 \frac{\mathrm{d}I_b(t)}{\mathrm{d}t}<0\right\},\\
&\Omega_2:=\left\{(M,I_b)\in\mathbb{R}_+^2:\frac{\mathrm{d}M(t)}{\mathrm{d}t}>0,
\frac{\mathrm{d}I_b(t)}{\mathrm{d}t}>0\right\},\\
&\Omega_3:=\left\{(M,I_b)\in\mathbb{R}_+^2:\frac{\mathrm{d}M(t)}{\mathrm{d}t}<0,
\frac{\mathrm{d}I_b(t)}{\mathrm{d}t}>0\right\},\\
&\Omega_4:=\left\{(M,I_b)\in\mathbb{R}_+^2:\frac{\mathrm{d}M(t)}{\mathrm{d}t}<0,
\frac{\mathrm{d}I_b(t)}{\mathrm{d}t}<0\right\}.
\end{aligned}
\end{equation}

To discuss the stability of this positive periodic solution of model \eqref{System-2}, we give the following definitions.

\begin{definition}[Orbital stability\cite{Simeonov}]
Trajectory $\Pi^+(Z_0,t_0)$ is said to be orbitally stable if for
any given $\varepsilon>0$, there exists a constant
$\delta=\delta(\varepsilon)>0$ such that for any other solution $Z^*(t)$ of
model (\ref{System-2}), $\rho(Z^*(t),\Pi^+(Z_0,t_0))<\varepsilon$ for
all $t>t_0$ when $\rho(Z^*(t_0),\Pi^+(Z_0,t_0))<\delta$.
\end{definition}

\begin{definition}[Orbitally asymptotical stability\cite{Simeonov}]
Trajectory $\Pi^+(Z_0,t_0)$ is said to be orbitally asymptotically stable
if it is orbitally stable, and there exists a constant $\eta>0$ such that for
any other solution $Z^*(t)$ of model \eqref{System-2},
$\lim_{t\rightarrow\infty}\rho(z^*(t),\Pi^+(Z_0,t_0))=0$ when
$\rho(Z^*(t_0),\Pi^+(Z_0,t_0))<\eta$.
\end{definition}

 $P, Q, \xi, \eta$ are continuous functions from $\mathbb{R}^2$ into $\mathbb{R}$, $\Gamma \subseteq \mathbb{R}^2$ denotes the impulsive set. We denote $M^+ =M(t^+)$, $I_b^+ =I_b(t^+)$, $M(t)=M$ and $I_b(t)=I_b$ for simplicity. For each point $Z(M, I_b)\in \Gamma$, $Z^+ =I(Z)=(M + \xi (M,I_b), I_b + \eta (M,I_b))=(M^+, I_b^+) \in \mathbb{R}^2$. $Z^+$ is called an impulsive point of $Z$.

Next, we consider the autonomous system with impulsive effects
\begin{equation}\label{System-3}
\left\{
\begin{aligned}
&\frac{\mathrm{d}x}{\mathrm{d}t}=f(x,y),\quad
\frac{\mathrm{d}y}{\mathrm{d}t}=g(x,y),\qquad &\varphi(x,y)\neq
0,\\
&\triangle x=\xi(x,y),\quad \triangle y=\eta(x,y),& \varphi(x,y)=0,
\end{aligned}
\right.
\end{equation}
where $f$ and $g$ are continuous differentiable  functions defined on $\mathbb{R}^2$ and $\varphi$ is a sufficiently smooth function with $\nabla \varphi\neq 0$. Let $(\widetilde{x}(t),\widetilde{y}(t))$ be a positive $T$-periodic solution of system \eqref{System-3}. The following result comes from corollary 2 of Theorem 1 of \cite{Simeonov}.

\begin{lemma}[Analogue of Poincar\'e criterion]\label{lemma2.1}
 If
the Floquet multiplier $\mu$ satisfies   $|\mu|<1$, where
$$
\mu=\prod_{j=1}^n\kappa_j\exp{\left\{\int_0^T\left[\frac{\partial
f(\widetilde{x}(t),\widetilde{y}(t))}{\partial x}+\frac{\partial
g(\widetilde{x}(t),\widetilde{y}(t))}{\partial y}\right]\,\mathrm{d}t\right\}}
$$
with
$$
\kappa_j=\frac{(\frac{\partial\eta}{\partial
y}\frac{\partial\varphi}{\partial x}-\frac{\partial\eta}{\partial
x}\frac{\partial\varphi}{\partial y}+\frac{\partial\varphi}{\partial
x})f_++(\frac{\partial\xi}{\partial
x}\frac{\partial\varphi}{\partial y}-\frac{\partial\xi}{\partial
y}\frac{\partial\varphi}{\partial x}+\frac{\partial\varphi}{\partial
y})g_+}{\frac{\partial\varphi}{\partial
x}f+\frac{\partial\varphi}{\partial y}g}
$$
and $f$, $g$, $\partial\xi/\partial x$, $\partial\xi/\partial y$, $\partial\eta/\partial x$, $\partial\eta/\partial y$, $\partial\varphi/\partial x$, and $\partial\varphi/\partial y$ have been calculated at the point $(\mu(\tau_j),\nu(\tau_j))$, $f_+=f(\mu(\tau_j^+),\nu(\tau_j^+))$, $g_+=g(\mu(\tau_j^+),\nu(\tau_j^+))$, and $\tau_j$ $(j\in N)$ is the time of the $j$-th jump. Then, $(\widetilde{x}(t),\widetilde{y}(t))$ is orbitally asymptotically stable.
\end{lemma}

\section{Stability analysis of equilibria for model (\ref{System-2}) without impulsive effect}

In this section, we study the qualitative characteristic of model \eqref{System-2} without impulsive effect. If impulsive effect is not introduced, then model \eqref{System-2} is
\begin{equation}\label{System-4}
\left\{
\begin{aligned}
&\frac{\mathrm{d}M(t)}{\mathrm{d}t}=\mu_m M(t)\left(1-\frac{M(t)}{K_m}\right)-\delta_m M(t):=P(M, I_b)\\
&\frac{\mathrm{d}I_b(t)}{\mathrm{d}t}=c\beta_{bm}\left(1-\frac{ I_b(t)}{N_b}\right)M(t)-\mu_{b}I_{b}(t):=Q(M, I_b).
\end{aligned}
\right.
\end{equation}

Clearly, if $\mu_m <\delta_m$, then model \eqref{System-4} only has a unique disease-free equilibrium $E(0,0)$; if $\mu_m > \delta_m$, then model \eqref{System-4} has a disease-free equilibrium $E(0,0)$ and a endemic equilibrium $E^*(M^*,I_b^*)$, where
$$
M^*=\frac{K_m(\mu_m-\delta_m)}{\mu_m},\qquad
I_b^*=\frac{c\beta_{bm}k_{m}N_b(\mu_m-\delta_m)}{c\beta_{bm}K_m(\mu_m-\delta_m)+\mu_m \mu_b N_b}.
$$

To investigate the stability of equilibrium, linearizing model \eqref{System-4} around the point $E(V_m, I_b)$ yields the Jacobian matrix
$$
J(V_m, I_b)=\begin{pmatrix}
\mu_m-\delta_m-\frac{2\mu_m M(t)}{N_b} &\quad 0\\
c\beta_{bm}\left(1-\frac{I_b(t)}{N_b}\right) &\quad -\frac{c\beta_{bm}M(t)}{N_b}-\mu_b
\end{pmatrix}.
$$
From the eigenvalues of the Jacobin matrix at equilibrium $E(0, 0)$ and $E^*(M^*, I_b^*)$ in respective cases, we easily obtain the following Proposition \ref{proposition-1}.
\begin{proposition}\label{proposition-1}
If $\mu_m>\delta_m$, then disease-free equilibrium $E(0, 0)$ of model \eqref{System-4} is a saddle point and the endemic equilibrium $E^*(M^*, I_b^*)$ of model \eqref{System-4} is a locally stable node point.
\end{proposition}

\begin{proposition}\label{proposition-3}
Model \eqref{System-4} has no closed orbit.
\end{proposition}
\begin{proof}
We choose Dulac function $B(M, I_b)=1/M$, which is continuously derivative function.
$$
\begin{aligned}
\frac{\partial (BP)}{\partial M}+\frac{\partial (BQ)}{\partial I_b}=&
-\frac{\mu_m}{K_m}-\frac{c\beta_{bm}}{N_b}-\mu_b < 0.
\end{aligned}
$$
According to the Bendixon Theorem \cite{Simeonov}, there is no closed orbit in the first quadrant. The proof is complete.
\end{proof}

From the above discussing, we can come to the following conclusion.
\begin{theorem}\label{Theorem-1}
Model \
If $\mu_m < \delta_m $, then model \eqref{System-4} admits only a globally asymptotically stable disease-free
equilibrium $E(0,0)$ and no endemic equilibrium. Further, if $\mu_m > \delta_m$, then model \eqref{System-4} exists an unstable equilibrium $E(0,0)$ and a unique globally asymptotically stable endemic equilibrium point $E^*(M^*, I_b^*)$.
\end{theorem}

\section{Dynamics of model (\ref{System-2}) with impulsive control}

In this section, we shall investigate the occurrence of periodic behavior for model \eqref{System-2}. Since virus equilibrium $(M^*,I_b^*)$ is globally asymptotically stable for
model \eqref{System-2} without impulsive control, any solution of model \eqref{System-2} without state
feedback control strategies will eventually tend to $(M^*,I_b^*)$. Therefore, we easily see that trajectory of model \eqref{System-2} with initial value $(N_{m0},I_{b0})\in \Omega$ will intersects impulsive set $\Gamma$ infinitely many times for $H_b<I_b^*$. Letting the coordinates of isocline $I_b'(t)=0$ intersecting line $I_b=(1-q)H_b$ and line $I_b=H_b$ is $H_q(N_{mq}, (1-q)H_b)$ and $H_h(N_{mh}, H_b)$ respectively, where
$$
N_{mq}=\frac{\mu_b N_b(1-q)H_b}{c\beta_{bm}(N_b-(1-q)H_b)}, \quad N_{mh}=\frac{\mu_b N_b H_b}{c\beta_{bm}(N_b-H_b)}.
$$
And assuming the point $H_h(N_{mh}, H_b)$ of impulsive effects are $H_h^+(N_{mh}^+, (1-q)H_b)$. Correspondingly, the coordinates of vertical isocline $\mathrm{d}M(t)/\mathrm{d}t=0$ hitting line $I_b=(1-q)H_b$ and line $I_b=H_b$ are $C_1^+(M^*, (1-q)H_b)$ and $C_2(M^*, H_b)$ respectively. In this subsection, we shall give some sufficient conditions for the existence and stability of positive periodic solutions in the case of $I_b\leq H_b$.

On the positivity of solutions of model \eqref{System-2}, we have, firstly, the following result
\begin{theorem}\label{Theorem-2}
Solutions of model \eqref{System-2}  with the initial value in $\mathbb{R}_+^2$ at time $t=t_0\geq 0$ are positive.
\end{theorem}

The proof of Theorem \ref{Theorem-2} is obvious, hence we omit it here.

The following result is the existence of positive order-1 periodic solution.
\begin{theorem}\label{Theorem-3}
For any $p, q\in(0,\;1)$ and $\mu_m>\delta_m$, model \eqref{System-2} admits a
positive order-1 periodic solution.
\end{theorem}
\begin{proof}

Letting the point $M_1^+(\varepsilon, (1-q)H_b)$ in domain $I\!I$ satisfies that $\varepsilon$ is small enough and $0 < \varepsilon <(1-p)N_{mh}$. The trajectory starting from the point $M_1^+$ in domain $I\!I$ will enter into domain $I\!I\!I$ and intersects section $\Omega$ at point $M_{1p}^+(\varepsilon_{1p}, (1-q)H_b)$, and then reaches section $\Gamma$ at point $M_2(V_{m2}, H_b)$, where $N_{m2}\in (N_{mh}, M^*)$. At point $M_2$, trajectory $\Pi^+(M_1^+,t_0)$ maps to the point $M_2^+((1-p)N_{m2}, (1-q)H_b)$ on phase set $\Omega$, due to impulsive effects $\triangle M(t)=-pM(t)$ and $\triangle I_b(t)=-qH_b$. Since $\varepsilon \leq (1-p)N_{mh}$, point $M_1^+$ is left of point $M_2^+$. Therefore, it follows $N_{m2}^+ =\digamma(N_{m1}^+)$ and the following inequality must hold true from Poincar\'e map (\ref{equation-1})
\begin{equation}\label{Equation-0013}
\digamma(N_{m1}^+)-N_{m1}^+=N_{m2}^+ - N_{m1}^+ >0
\end{equation}

On the other hand, the vertical line $\mathrm{d}M(t)/\mathrm{d}t=0$ intersects line $I_b=(1-q)H_b$ and line $I_b=H_b$ at points $C_1^+(M^*, (1-q)H_b)$ and $C_2(M^*, H_b)$, respectively. Supposing that trajectory $\Pi^+(C_1^+,t_0)$ starting from the initial point $C_1^+$ reaches impulsive set $\Gamma$ at point $C_2(M^*, H_b)$ and jumps to point ${C_2}^+((1-p)M^*, (1-q)H_b)$. Obviously, we have point ${C_2}^+$ is left of point $C_1^+$. Thus, the following inequality holds
\begin{equation}\label{Equation-003}
\digamma(N_{m3}^+)-N_{m3}^+=N_{m4}^+ - N_{m3}^+ <0
\end{equation}
This together with \eqref{Equation-0013} gives that Poincar\'e map $\digamma$ has a fixed point, which corresponds to an order-1 periodic solution of model \eqref{System-2}. The proof is complete.
\end{proof}

Nextly, on the orbital stability of this periodic
solution of model \eqref{System-2}, we have the following
Theorem.

\begin{theorem}\label{Theorem-4}
Let $(\mu(t),\nu(t))$ be a positive order-1 periodic solution of model \eqref{System-2} with period $T$. If
\begin{equation*}\label{Equation-3}
\left|\mu\right|=\left|\kappa_1\right|
\exp\left\{-\int_0^T\left[\frac{\mu_m}{K_m}\mu(t)
+\mu_b+\frac{c\beta_{bm}}{N_b}\mu(t)\right]\,\mathrm{d} t\right\}<1,
\end{equation*}
where
\begin{equation*}\label{Equation-03}
\kappa_1=\frac{c\beta_{bm}[N_b-(1-q)H_b](1-p)\mu(T)
-\mu_b(1-q)H_bN_b}{c\beta_{bm}(N_b-H_b)\mu(T)-\mu_b H_b N_b}
\end{equation*}
then $(\mu(t),\nu(t))$ is orbitally asymptotically stable.
\end{theorem}
\begin{proof}
Assuming that the order-1 periodic solution $(\mu(t),\nu(t))$ with period $T$ intersects the impulsive set $\Gamma$ and phase set $\Omega$ at points $M(\mu(T),H_b)$ and
$M^+((1-p)\mu(T),(1-q)H_b)$, respectively. Comparing with system (\ref{System-3}), we have
\begin{equation}\label{Equation-04}
P(M, I_b)=\mu_m M(1-{M}/{K_m})-\delta_m M, \quad Q(M, I_b)=c\beta_{bm}(1-{ I_b}/{N_b})M-\mu_{b}I_{b},
\end{equation}
\begin{equation}\label{Equation-04}
\xi(M,I_b)=-pM,\quad \eta(M,I_b)=-qI_b,\quad\varphi(M,I_b)=I_b-H_b,\\
\end{equation}
and $$(\mu(T),\nu(T))=(\mu(T),H_b)\quad (\mu(T^+),\nu(T^+))=((1-p)\mu(T),(1-q)H_b).
$$
Thus
\begin{equation}\label{Equation-4}
\frac{\partial P}{\partial M}=\mu_m-\delta_m-2M\frac{\mu_m}{K_m},\quad
\frac{\partial Q}{\partial I_b}=-\mu_b-c\beta_{bm}\frac{M}{N_b},
\end{equation}
and
\begin{equation}\label{Equation-5}
\frac{\partial\xi}{\partial M}=-p,\quad
\frac{\partial\eta}{\partial I_b}=-q,\quad
\frac{\partial\varphi}{\partial I_b}=1,\quad
\frac{\partial\xi}{\partial I_b}
=\frac{\partial\eta}{\partial M}
=\frac{\partial\varphi}{\partial M}=0.
\end{equation}
Furthermore, it follows from (\ref{Equation-4}) and (\ref{Equation-5}) that
\begin{equation}\label{Equation-6}
\begin{aligned}
\kappa=&\frac{\left(\frac{\partial\eta}{\partial
I_b}\frac{\partial\varphi}{\partial M}-\frac{\partial\eta}{\partial
M}\frac{\partial\varphi}{\partial I_b}+\frac{\partial\varphi}{\partial
M}\right)P_++\left(\frac{\partial\xi}{\partial
M}\frac{\partial\varphi}{\partial I_b}-\frac{\partial\xi}{\partial
I_b}\frac{\partial\varphi}{\partial M}+\frac{\partial\varphi}{\partial
I_b}\right)Q_+}{\frac{\partial\varphi}{\partial
M}P+\frac{\partial\varphi}{\partial I_b}Q} \\
=&\frac{(1-p)Q(\mu(T^+),\nu(T^+))}{Q(\mu(T),\nu(T))}\\
=&(1-p)\frac{c\beta_{bm}[N_b-(1-q)H_b](1-p)\mu(T)
-\mu_b(1-q)H_bN_b}{c\beta_{bm}(N_b-H_b)\mu(T)-\mu_bH_bN_b}
\end{aligned}
\end{equation}
and
\begin{equation}\label{Equation-7}
\mu=\kappa\exp\left\{\int_0^T
\left(\mu_m-\delta_m-2\frac{\mu_m}{K_m}\mu(t)
-\mu_b-\frac{c\beta_{bm}}{N_b}\mu(t)\right)\,\mathrm{d} t\right\}.
\end{equation}

On the other hand, integrating the both sides of the first
equation of model (\ref{System-3}) along the orbit
$\widehat{M^+M}$ to give
\begin{equation}\label{Equation-8}
\ln\frac{1}{1-p}=\int_{(1-p)\mu(T)}^{\mu(T)}\frac{\mathrm{d}M}{M}=\int_0^T
[\mu_m-\delta_m-\frac{\mu_m}{K_m}\mu(t)]\,\mathrm{d}t.
\end{equation}
From (\ref{Equation-6})-(\ref{Equation-8}), we obtain
$$
\begin{aligned}
|\mu|=&
\frac{1}{1-p} \left |(1-p)\frac{c\beta_{bm}[N_b-(1-q)H_b](1-p)\mu(T)
-\mu_b(1-q)H_bN_b}{c\beta_{bm}(N_b-H_b)\mu(T)-\mu_bH_bN_b}\right|
\\
&\times \exp\left\{-\int_0^T\left(\frac{\mu_m}{K_m}\mu(t)
+\mu_b+\frac{c\beta_{bm}}{N_b}\mu(t)\right)\,\mathrm{d}t\right\}\\
=&\left|\frac{c\beta_{bm}[N_b-(1-q)H_b](1-p)\mu(T)
-\mu_b(1-q)H_bN_b}{c\beta_{bm}(N_b-H_b)\mu(T)-\mu_bH_bN_b}\right|\\
&\times\exp\left\{-\int_0^T\left(\frac{\mu_m}{K_m}\mu(t)
+\mu_b+\frac{c\beta_{bm}}{N_b}\mu(t)\right)\,\mathrm{d}t\right\}.
\end{aligned}
$$

From what has been discussed above, we know that model \eqref{System-2}
satisfies all conditions of Lemma \ref{lemma2.1}. It then follows
from Lemma \ref{lemma2.1} that the order-1 periodic solution
$(\mu(t),\nu(t))$ of model \eqref{System-2} is orbitally
asymptotically stable and has asymptotic phase property. This
completes the proof.
\end{proof}

\begin{remark}
Generally, the condition of Theorem \ref{Theorem-2} is not
easy to test since we cannot compute explicitly the analytic expressions of
the positive order-1 periodic solution $(\mu(t),\nu(t))$. However, we shall give in what follows two broad situations which correspond to the following cases.
\end{remark}

\begin{corollary}\label{Corollary-1}
If
$$
k_1=\left|\frac{c\beta_{bm}[N_b-(1-q)H_b](1-p)\mu(T)
-\mu_b(1-q)H_bN_b}{c\beta_{bm}(N_b-H_b)\mu(T)-\mu_bH_bN_b}\right|<1,
$$
then order-1 periodic solution $(\mu(t),\nu(t))$ of model \eqref{System-2} is orbitally asymptotically stable.
\end{corollary}

Finally, considering the position relation between isoclinics and impulsive $\Gamma$ or phase set $\Omega$, we state and prove our result on the existence and stability of positive order-$k(k=1,2)$ periodic solutions of
model \eqref{System-2} in different cases.

\begin{itemize}
\item The case of $(1-p)M^*<N_{mq}$
\end{itemize}

\begin{theorem}\label{Theorem-5}
If $(1-p)M^*<N_{mq}$, then model \eqref{System-2} has a uniquely positive order-1 or order-2 periodic solution, which is orbitally asymptotically stable. Further, there is no order-$k(k\geq3)$ periodic solution.
\end{theorem}
\begin{proof}
Let $A_0^+(N_{m0}^+, I_{b0}^+) \in \Omega$ in domain $I$. According to the geometrical structure of the phase field of model \eqref{System-2}, trajectory $\Pi^+(A_0^+, t_0)$ initiating from point $A_0^+$ turns around point $H_q$ and reach impulsive set $\Gamma$ at point $A_1(N_{m1}, I_{b1})$, and then map to phase set $\Omega$ at point $A_1^+(N_{m1}^+, I_{b1}^+)$ after one time impulsive effect. Repeating the above process, we have $N_{m(n+1)}^+ =\digamma(N_{mn}^+ ), n=1,2,3\cdots.$ Based on the condition $(1-p)M^* < N_{mq}$, we have point $A_1^+$
is left of point $H_q$. If $\widetilde{A}_0^+ =\widetilde{A}_1^+$, then model \eqref{System-2} exists a
positive order-1 periodic solution. And if $\widetilde{A}_0^+\neq \widetilde{A}_1^+$, and $\widetilde{A}_0^+ = \widetilde{A}_2^+$, then model \eqref{System-2} has a positive order-2 periodic solution.

 Next, we discuss the general circumstance, that is $\widetilde{A}_0^+ \neq \widetilde{A}_2^+ \neq\cdots\neq \widetilde{A}_n^+ (n>2).$ For the position relation of $\widetilde{A}_0^+$, $\widetilde{A}_1^+$ and $\widetilde{A}_2^+$, it has the following possible situations.
\item[(a).] If $\widetilde{A}_0^+ < \widetilde{A}_1^+$, there are two possible cases that are $\widetilde{A}_0^+ < \widetilde{A}_2^+< \widetilde{A}_1^+$ and $\widetilde{A}_2^+ < \widetilde{A}_0^+< \widetilde{A}_1^+.$ \\
$(i).$ If $\widetilde{A}_0^+ < \widetilde{A}_2^+< \widetilde{A}_1^+$, we have
$$
0<\widetilde{A}_0^+ < \widetilde{A}_2^+< \widetilde{A}_4^+ < \cdots<\widetilde{A}_{2k}^+ <\cdots < \widetilde{A}_{2k+1}^+< \cdots< \widetilde{A}_3^+< \widetilde{A}_1^+ < H_q,
$$
That is
$$
0<N_{m0}^+<N_{m2}^+<N_{m4}^+<\cdots<N_{m(2k)}^+< \cdots<N_{m(2k+1)}^+<\cdots<N_{m3}^+<N_{m1}^+<N_{mq}.
$$

$(ii).$ If $\widetilde{A}_2^+ < \widetilde{A}_0^+< \widetilde{A}_1^+$, we yield
$$
0<\cdots < \widetilde{A}_{2k}^+ < \cdots <\widetilde{A}_4^+ <\widetilde{A}_2^+ < \widetilde{A}_0^+< \widetilde{A}_1^+ < \widetilde{A}_3^+ <\cdots\widetilde{A}_{2k+1}^+<\cdots<H_q,
$$
That is
$$
0<\cdots<N_{m(2k)}^+<\cdots<N_{m4}^+<N_{m2}^+<N_{m0}^+<N_{m1}^+<N_{m3}^+
<\cdots<N_{m(2k+1)}^+<\cdots<N_{mq}.
$$
\item[(b).] If $\widetilde{A}_1^+ < \widetilde{A}_0^+$, there are two possible cases $\widetilde{A}_1^+ < \widetilde{A}_0^+< \widetilde{A}_2^+$ and $\widetilde{A}_1^+ < \widetilde{A}_2^+< \widetilde{A}_0^+.$\\
$(iii).$ If $\widetilde{A}_1^+ < \widetilde{A}_0^+ < \widetilde{A}_2^+$, we get
$$
0<\cdots < \widetilde{A}_{2k+1}^+< \cdots < \widetilde{A}_3^+ <\widetilde{A}_1^+ < \widetilde{A}_0^+< \widetilde{A}_2^+< \widetilde{A}_4^+<\cdots\widetilde{A}_{2k}^+<\cdots<H_q,
$$
That is
$$
0<\cdots<N_{m(2k+1)}^+<\cdots<N_{m3}^+<N_{m1}^+<N_{m0}^+<N_{m2}^+<N_{m4}^+<\cdots<N_{m(2k)}^+
<\cdots<N_{mq}.
$$
$(iv).$ If $\widetilde{A}_1^+ < \widetilde{A}_2^+ < \widetilde{A}_0^+$, we have
$$0<\widetilde{A}_1^+ < \widetilde{A}_3^+ < \cdots< \widetilde{A}_{2k+1}^+ < \cdots < \widetilde{A}_{2k}^+ < \cdots< \widetilde{A}_4^+ < \widetilde{A}_2^+ <\widetilde{A}_0^+ <H_q,$$
That is,
$$
0<N_{m1}^+<N_{m3}^+<\cdots<N_{m(2k+1)}^+< \cdots<N_{m(2k)}^+<\cdots<N_{m4}^+<N_{m2}^+<N_{m0}^+<N_{mq}.
$$
\begin{figure}[htb]
\centering
\subfiguretopcaptrue
\subfigure[]{
\label{Fig.1(a)}
\includegraphics[width=0.46\textwidth]{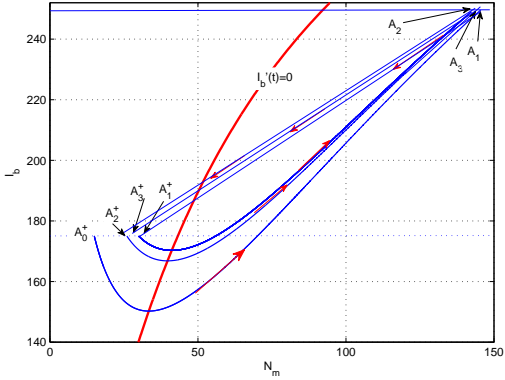}}
\subfigure[]{
\label{Fig.1(b)}
\includegraphics[width=0.46\textwidth]{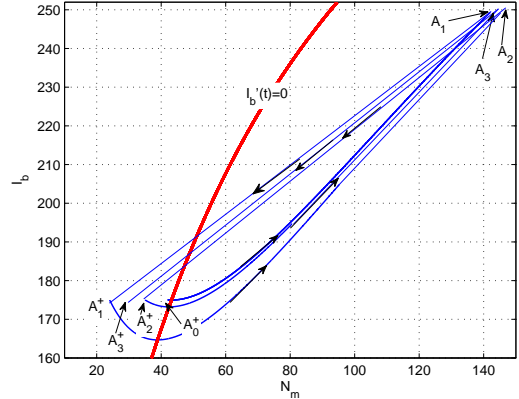}}
\caption{ The illustration of existence of order-1 periodic solution of model  \eqref{System-2}.}
\label{Fig.1}
\end{figure}
On the basis of the above analysis, in $(ii)$ of case $(a)$, it follows that $\lim _{k\rightarrow \infty}N_{m(2k)}^+=N_2^*$ and $\lim _{k\rightarrow \infty}N_{m(2k+1)}^+=N_1^*$, where $0<N_2^*<N_1^*\leq N_{mq}$. And then, we have $N_1^*=\digamma(N_2^*)$, $N_2^*=\digamma(N_1^*)$. Therefore, model \eqref{System-2} has an orbitally asymptotically stable positive order-2 periodic solution. Similarly, in $(iii)$ of case (b), model \eqref{System-2} has an orbitally asymptotically stable positive order-2 periodic solution. For $(i)$ of case (a) and $(iv)$ of case (b), there exists a $N^*\in(0, Mq)$ such that $\lim _{k\rightarrow \infty}N_{m(2k)}^+=\lim _{k\rightarrow \infty}N_{m(2k+1)}^+=N^*$. This conclusion implies that there exists an order-1 periodic solution of model \eqref{System-2}. This proof is complete.
\end{proof}

\begin{itemize}
\item The case of $(1-p)N_{mh}>N_{mq}$
\end{itemize}

\begin{theorem}\label{Theorem-6}
If $(1-p)N_{mh}>N_{mq}$, then model \eqref{System-2} exists a uniquely positive order-1 periodic solution, which is orbitally asymptotically stable. Further, there is no order-$k(k\geq2)$ periodic solution.
\end{theorem}
\begin{proof}
According to the geometrical structure of the phase field of model \eqref{System-2}, trajectory $\Pi^+(B_1^+, t_0)$ from point $B_1^+(N_{m1}^+, (1-q)H_b)\in\Omega$ in domain $I$ turns around point $H_q$, intersects impulsive set $\Gamma$ at point $B_2(N_{m2}, H_b)$ and then map to point $B_2^+(N_{m2}^+, (1-q)H_b)$, where $N_{m2}^+ =(1-p)N_{m2}=\digamma(N_{m1}^+)$, by Poincar$\acute{e}$ map \eqref{equation-1}. Based on the condition $(1-p)N_{mh}>N_{mq}$, we yield that point $B_2^+$ is right of the point $H_q$. Further, $N_{m2}^+ - N_{m1}^+ =\digamma(N_{m1}^+)-N_{m1}^+ >0$.

On the other hand, the trajectory $\Pi^+(C_1^+, t_0)$ starting from point $C_1^+(M^*,(1-q)H_b)$ reaches point $C_2(M^*, H_b)$, and then jumps to point $C_2^+(N_{2m}^+, (1-q)H_b)\in\Omega$, where $N_{2m}^+ =(1-p)M^*\in(N_{mq}, M^*)$. Obviously, $N_{2m}^+ - M^*=\digamma(M^*)-M^*<0$, Therefore, there must exist a point $D_1^+(N_{dm}, (1-q)H_b)\in \Omega$ between $B_1^+$ and $C_1^+$ such that $N_{dm}^+ - N_{dm}=\digamma(N_{dm}) - N_{dm}=0$. That is to say, the trajectory $\Pi^+(D_1^+, t_0)$ starting from the point $D_1^+(N_{md}^+, (1-q)H_b)$ exists an order-1 periodic solution of model \eqref{System-2}.

Since trajectory starting from points in sets $\{(M, I_b): 0<M<N_{m1}^+, I_b=(1-q)H_b\}$ and $\{(M, I_b): M>M^*, I_b=(1-q)H_b\}$ will enter set $\overline{B_1^+C_1^+}=\{(M, I_b): N_{m1}^+ < M < M^*, I_b=(1-q)H_b\}$ expericing impulsive effects after several times at most. Therefore, the initial point of the order-1 periodic solution only lies in set $\overline{B_1^+C_1^+}$. The set $\overline{B_1^+C_1^+}=\{(M, I_b): N_{m1}^+ < M < M^*, I_b=(1-q)H_b\}$ is mapped to the set $\overline{B_2 C_2}=\{(M, I_b): N_{mh} < M < M^*, I_b=(1-q)H_b\}$ by the first and second equations of model (\ref{System-3}). Subsequently, the set $\overline{B_2 C_2}$ is mapped to the set $\overline{B_2^+ C_2^+}=\{(M, I_b): N_{m2}^+ < M < N_{2m}^+, I_b=(1-q)H_b\}$ by the third and fourth equations of model \eqref{System-2}. It is to know $V_{m1}^+ < N_{m2}^+ < N_{2m}^+ < M^*$, the line segments $\mid\overline{B_1^+C_1^+}\mid > \mid\overline{B_2^+C_2^+}\mid > \cdots$. Continuing the above process from the vector field of model \eqref{System-2}, it yields that
$$
\mid\overline{B_1^+C_1^+}\mid>\mid\overline{B_2^+C_2^+}\mid>\mid\overline{ B_3^+C_3^+}\mid>\cdots
$$
and
$$
0<N_{m1}^+ < N_{m2}^+ < N_{m3}^+ < \cdots < N_{3m}^+ < N_{2m}^+ <N_{1m}^+ <M^*
$$
Further, the sequence $\mid\overline{B_n C_n}\mid$ is monotonously  convergent and $\lim_{n\to\infty} \mid\overline{ B_n C_n}\mid=0$, which implies that there exists a unique point $D_1^+(N_{md}^+, (1-q)H_b)$ such  that $\digamma(N_{md}^+)=N_{dm}^+$ and $\digamma(N_{dm}^+)=N_{md}^+$. That is to say, $\lim_{n\rightarrow \infty} N_{mn}^+=\lim_{n\to\infty} N_{nm}^+ =N_{md}^+= N_{dm}^+$.

Next, we demonstrate that the uniquely positive order-1 periodic solution is orbitally asymptotically stable. For any point $G(N_G, (1-q)H_b)\in \Omega$, where
$$
N_G\in[0,N_{m1}^+]\cup[N_{mn}^+, N_{m(n+1)}^+]\cup[N_{(n+1)m}^+, N_{nm}^+]\cup[N_{1m}^+]\cup[M^*,\infty ), n=1, 2,\cdots.
$$
Without loss of generality, setting $N_G \in [N_{mn}^+, N_{m(n+1)}^+]$. The trajectory $O^+(G,t_0)$ starting from the initial point $G(N_G, (1-q)H_b)$ intersects section $\Gamma$, next immediately jumps to point $G_1(N_{G_1},(1-q)H_b)\in \Omega$, where $N_{G_1}\in [N_{m(n+1)}^+, N_{m(n+2)}^+]$ due to impulsive effects. Repeating the above process, we obtain a sequence $\{G_n(N_{G_n}, (1-q)H_b)\}$ in section $\Omega$, where $V_{G_n} \in [N_{G_{n+k}},N_{G_{n+k+1}}] (k=1,2,\ldots)$. So $\lim_{n\rightarrow \infty} N_{G_n}=\lim_{n\rightarrow \infty} N_{mn}^+=V_{md}^+$. Similarly to, if $N_G \in [N_{(n+1)m}^+, N_{nm}^+]$, we also can get $\lim_{n\rightarrow \infty} N_{nm}^+=N_{dm}^+$. Furthermore, the trajectory from any initial point $G\in \Omega$, eventually tends to be the positive order-1 periodic solution.

From the above discussion, we obtain that model \eqref{System-2} has an unique order-1 periodic solution which is orbitally asymptotically stable for $(1-p)N_{mh}>N_{mq}$. This completes the proof.

\end{proof}

\begin{figure}[htb]
\centering
\subfiguretopcaptrue
\subfigure[]{
\label{Fig.2(a)}
\includegraphics[width=0.60\textwidth]{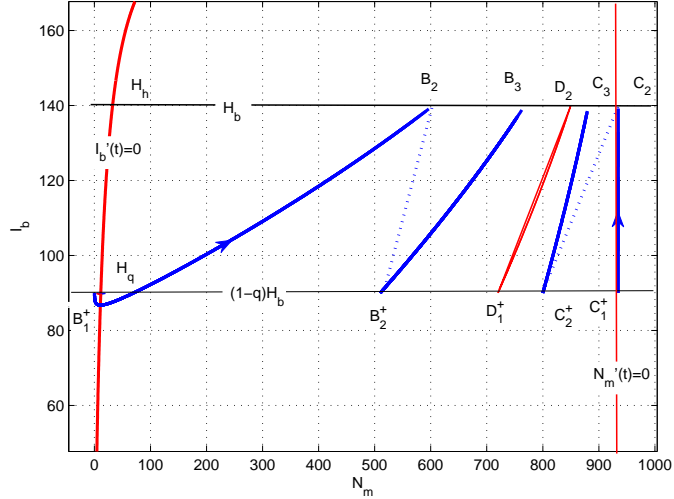}}
\caption{ The illustration of existence of order-1 periodic solution of model \eqref{System-2}.}
\label{Fig.2}
\end{figure}
\section{Numerical simulation and discussion}
In order to testify the validity of our results and the suitability of the impulsive state dependent pulse control strategy,
we consider the following WNV model with state dependent pulse control strategy.

\begin{figure}[htbp]
\centering
\subfiguretopcaptrue
\subfigure[]{
\label{Fig.3(a)}
\includegraphics[width=0.46\textwidth,height=0.24\textheight]{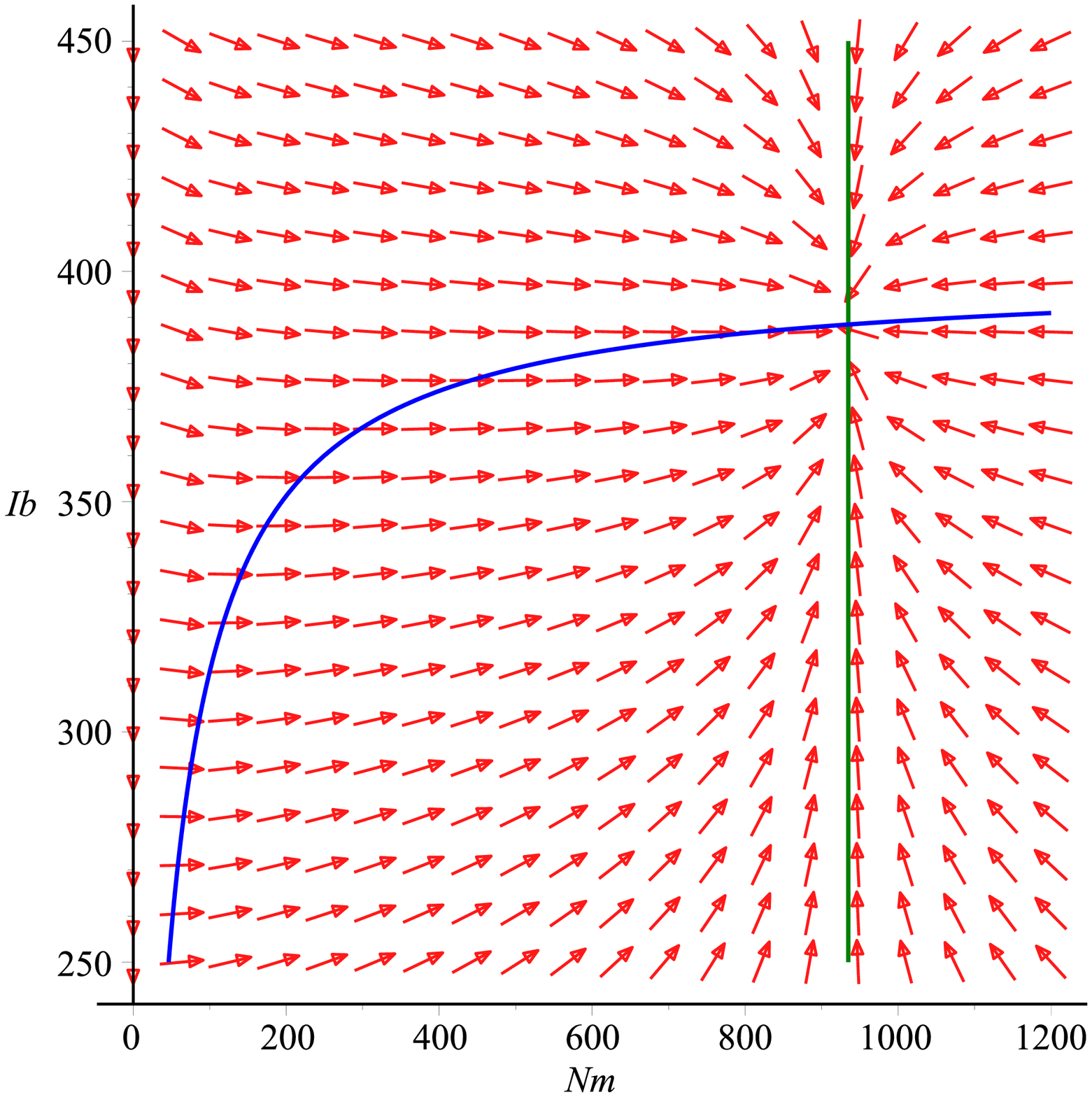}}
\subfigure[]{
\label{Fig.3(b)}
\includegraphics[width=0.46\textwidth]{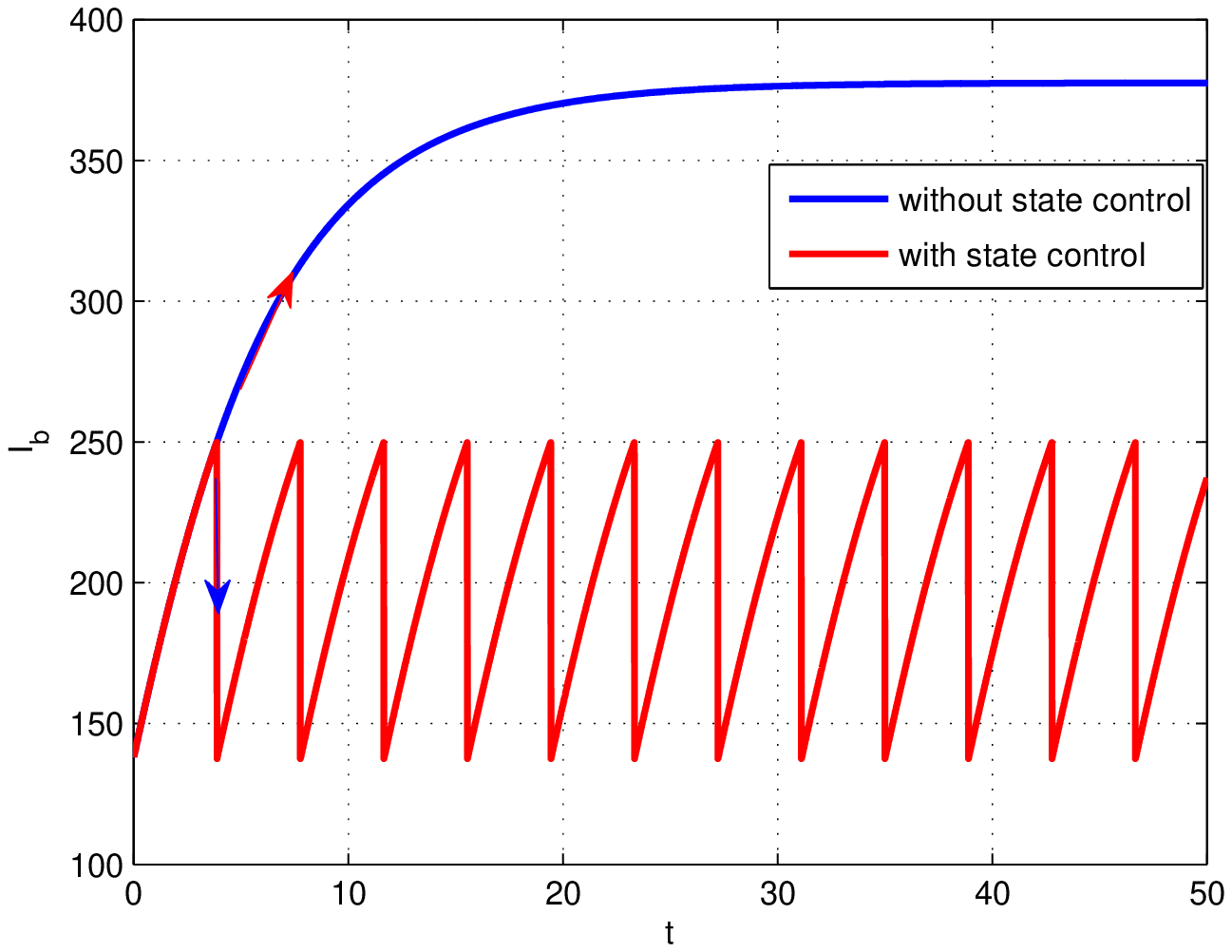}}
\caption{(a)Vector graph of model \eqref{System-2} with $\mu_m=0.537 > \delta_m=0.035$, $K_m=1000$, $c=0.09$, $\beta_{bm}=0.8$, $N_b=400$ and $\mu_b=0.01$; (b) There is a comparison of time series in $I_b(t)$ of model \eqref{System-2} between without state feedback control and state feedback control with $\mu_m=0.537$, $\delta_m=0.035$, $K_m=1000$, $c=0.09$, $\beta_{bm}=0.8$, $N_b=400$, $\mu_b=0.01$, $H_b=250<400$, $p=0.15$, $q=0.45$ and the initial value $N_{m0}=771$, $I_{b0}=137$.}
\label{Fig.1}
\end{figure}
Firstly, model \eqref{System-2} without pulse effects has a unique globally asymptotically stable endemic equilibrium point $E^*(V_m^*, I_b^*)$ =(934.23, 398.81), which is illustrated as Figure \ref{Fig.3(a)} by red line. Green line represents vertical isocline $M'(t)=0$, blue line represents horizontal isocline $I_b'(t)=0$. It is an obvious comparison that the time series in  of model \eqref{System-2} without state feedback control trends a unique globally asymptotically stable endemic equilibrium $E^*$. However, model \eqref{System-2} with state feedback control trends a stable state. That is to say, the number of the infected birds and mosquitoes is within a certain range by employing strategies of state-dependent feedback control.

\begin{figure}[htbp]
\centering
\subfiguretopcaptrue
\subfigure[]{
\label{Fig.4(a)}
\includegraphics[width=0.46\textwidth]{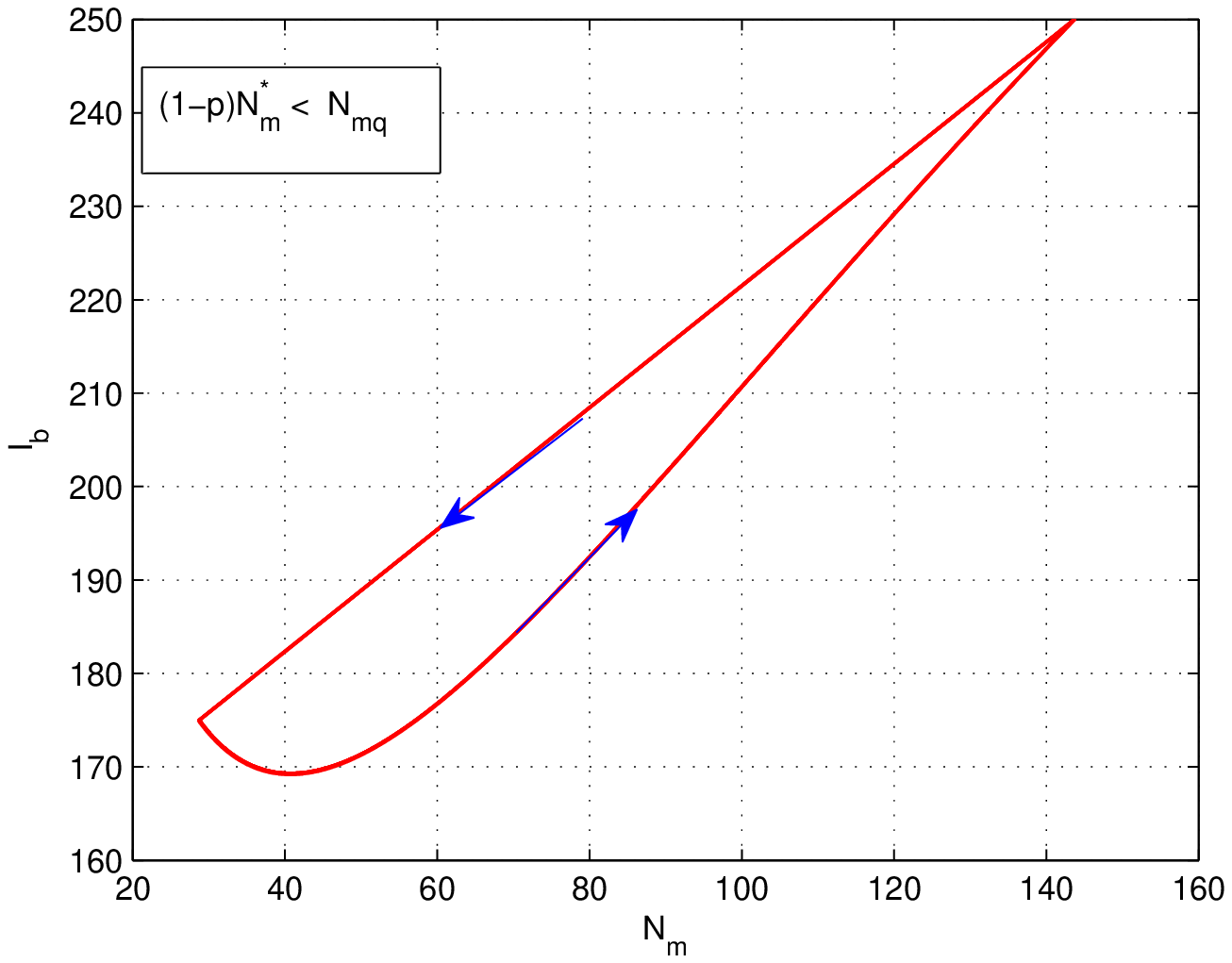}}
\subfigure[]{
\label{Fig.4(b)}
\includegraphics[width=0.46\textwidth]{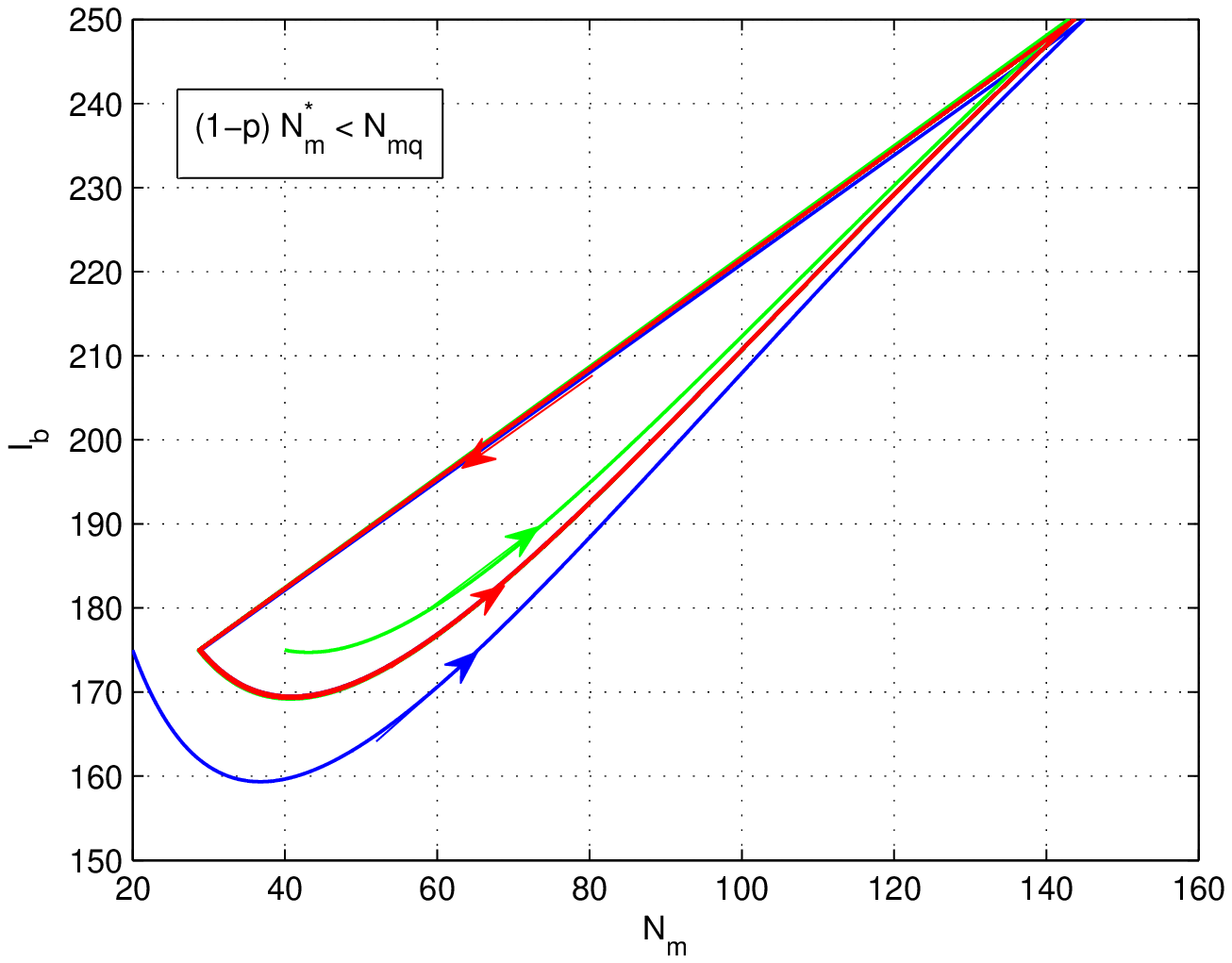}}
\caption{(a)The existence of order-1 periodic solution of model \eqref{System-2}. (b)The orbitally asymptotical stability of order-1 periodic solution of model \eqref{System-2}. The baseline parameter values are as follows:  $\mu_m=0.06$, $\delta_m=0.04$, $K_m=1000$, $c=0.09$, $\beta_{bm}=0.8$, $N_b=400$, $\mu_b=0.01$, $p=0.8$, $q=0.3$, $H_b=250<400$ and the initial value $N_{m0}=29$, $I_{b0}=175$ in the case of $(1-p)M^* < N_{mq}$.}
\label{Fig.4}
\end{figure}

Secondly, the existence and orbitally asymptotical stability of order-1 periodic solution of model \eqref{System-2} is shown in Figure \ref{Fig.4(a)} and \ref{Fig.4(b)} in the case of $(1-p)M^* < N_{mq}$. We know that the numerical simulation result is consistent with the therical reult of Theorem \eqref{Theorem-5}. Further, when the infection of West Nile Virus would get worse, we have to take this strategy of state feedback control and utlize this condition $(1-p)M^* < N_{mq}$ so that the number of mosquitoes should confine in a lower range.

\begin{figure}[htbp]
\centering
\subfiguretopcaptrue
\subfigure[]{
\label{Fig.5(a)}
\includegraphics[width=0.46\textwidth]{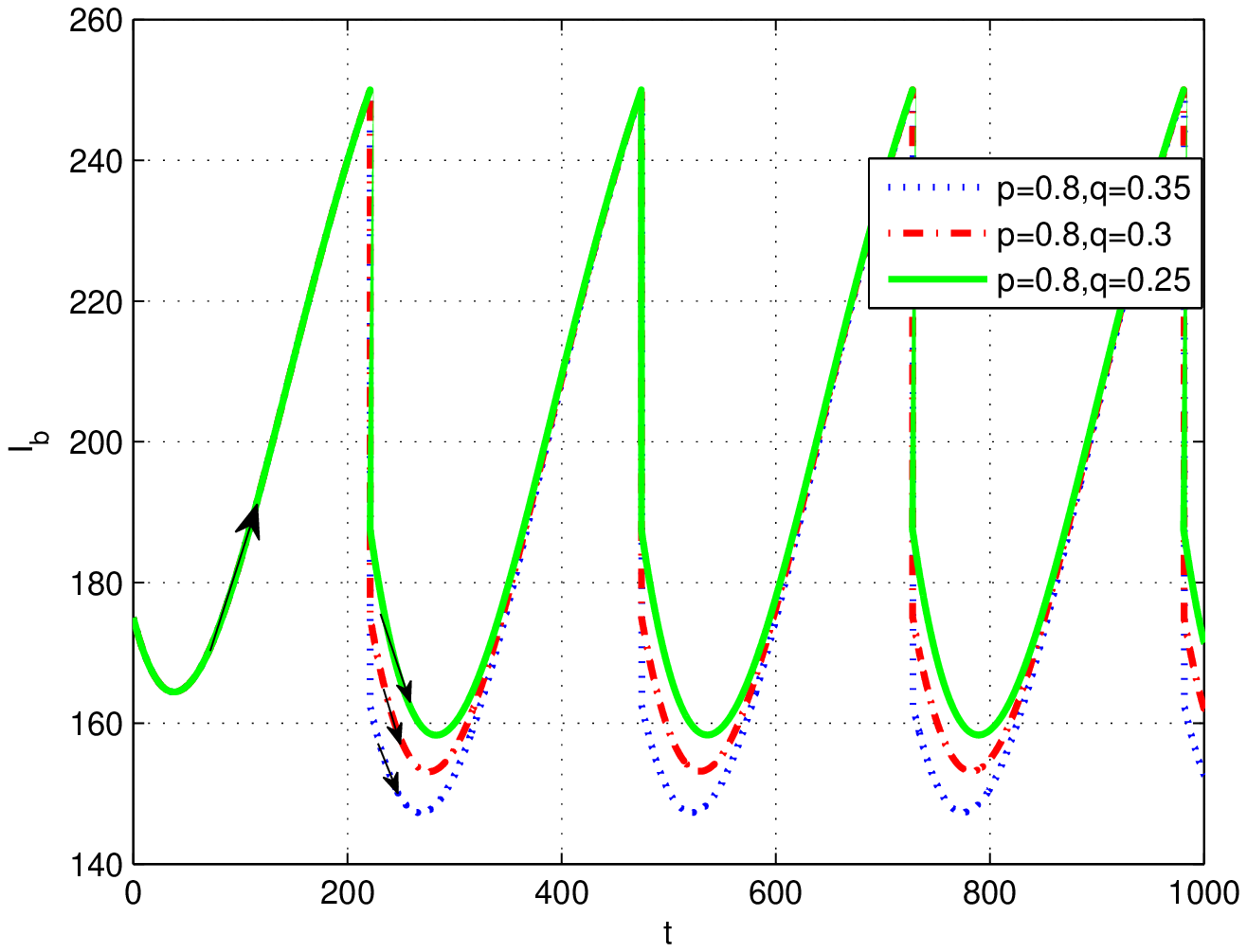}}
\subfigure[]{
\label{Fig.5(b)}
\includegraphics[width=0.46\textwidth]{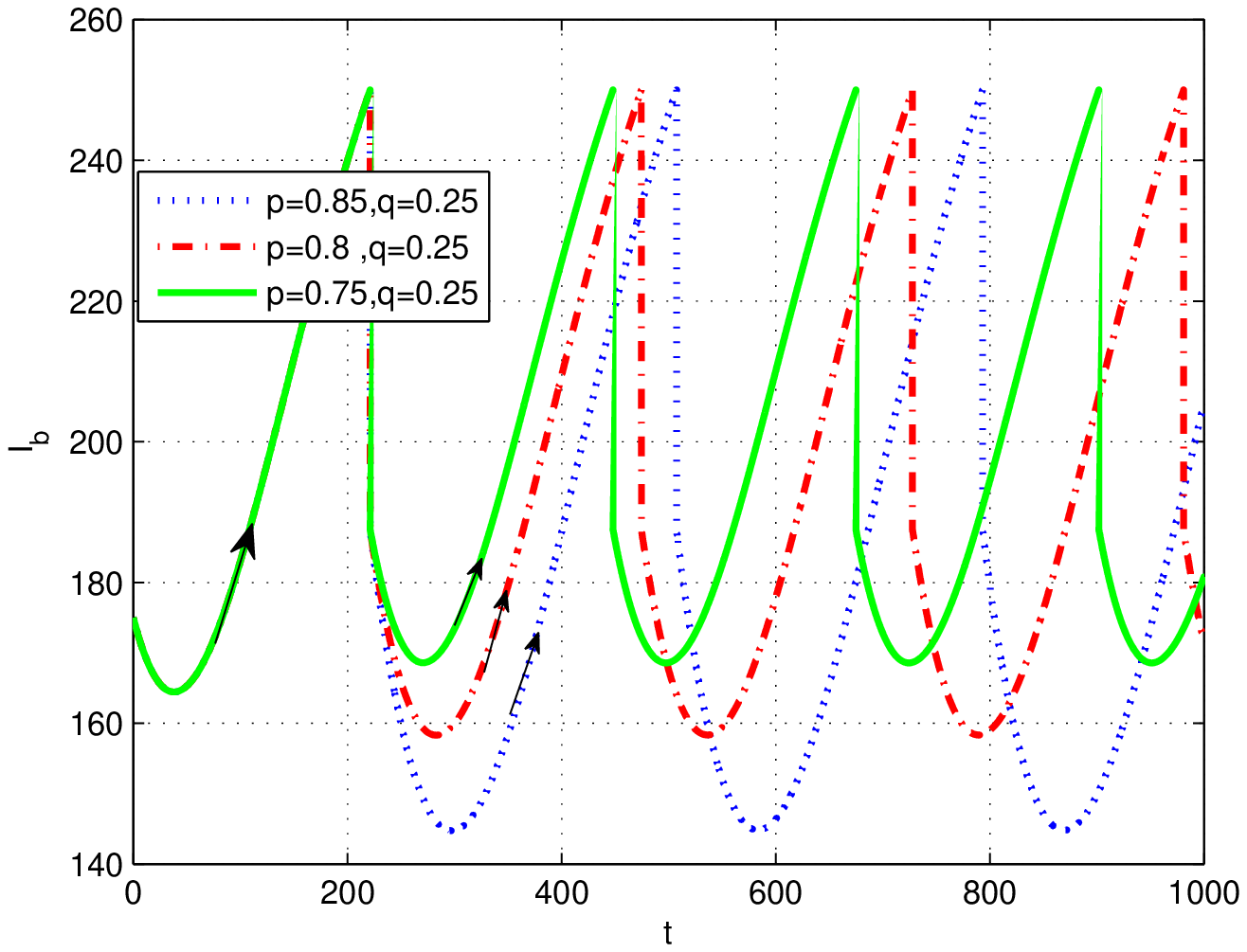}}
\caption{(a)The effects of $q$(the proportion of curing infected birds), where $q=0.35, 0.3, 0.25$ and $p=0.8$. (b)The effects of $p$(the proportion of culling mosquitoes), where $p=0.85, 0.8, 0.75$ and $q=0.25$. The other parameters are $\mu_m=0.06$, $\delta_m=0.05$, $K_m=1000$, $c=0.09$, $\beta_{bm}=0.8$, $N_b=400$, $\mu_b=0.01$, $H_b=250<400$,  and the initial value $N_{m0}=29$, $I_{b0}=175$ in the case of $(1-p)M^* < N_{mq}$.}
\label{Fig.5}
\end{figure}

We find that the number of infected birds $I_b$ dramatically decreases as curing rate $q$ increases while the time interval for  the infected birds reaching threshold value $H_b$ is fixed (shown in Figure \ref{Fig.5(a)}). From Figure \ref{Fig.5(b)}, we draw that the number of infected birds $I_b$ dramatically decreases, when $p$ increased. Further, the change of $p$ lead to the change of periodic. And then, we show that the periodic of infected birds increases as the cofficient $q$ of curing infected birds increased. Consequently, the culling rate of mosquitoes $p$ plays an important role in controlling the interval for the infected birds reaching threshold value $H_b$. The curing rate of birds $q$ is proportionate to the maximum amplitude of the number of infected birds $I_b$.

\begin{figure}[htbp]
\centering
\subfiguretopcaptrue
\subfigure[]{
\label{Fig.6(a)}
\includegraphics[width=0.46\textwidth]{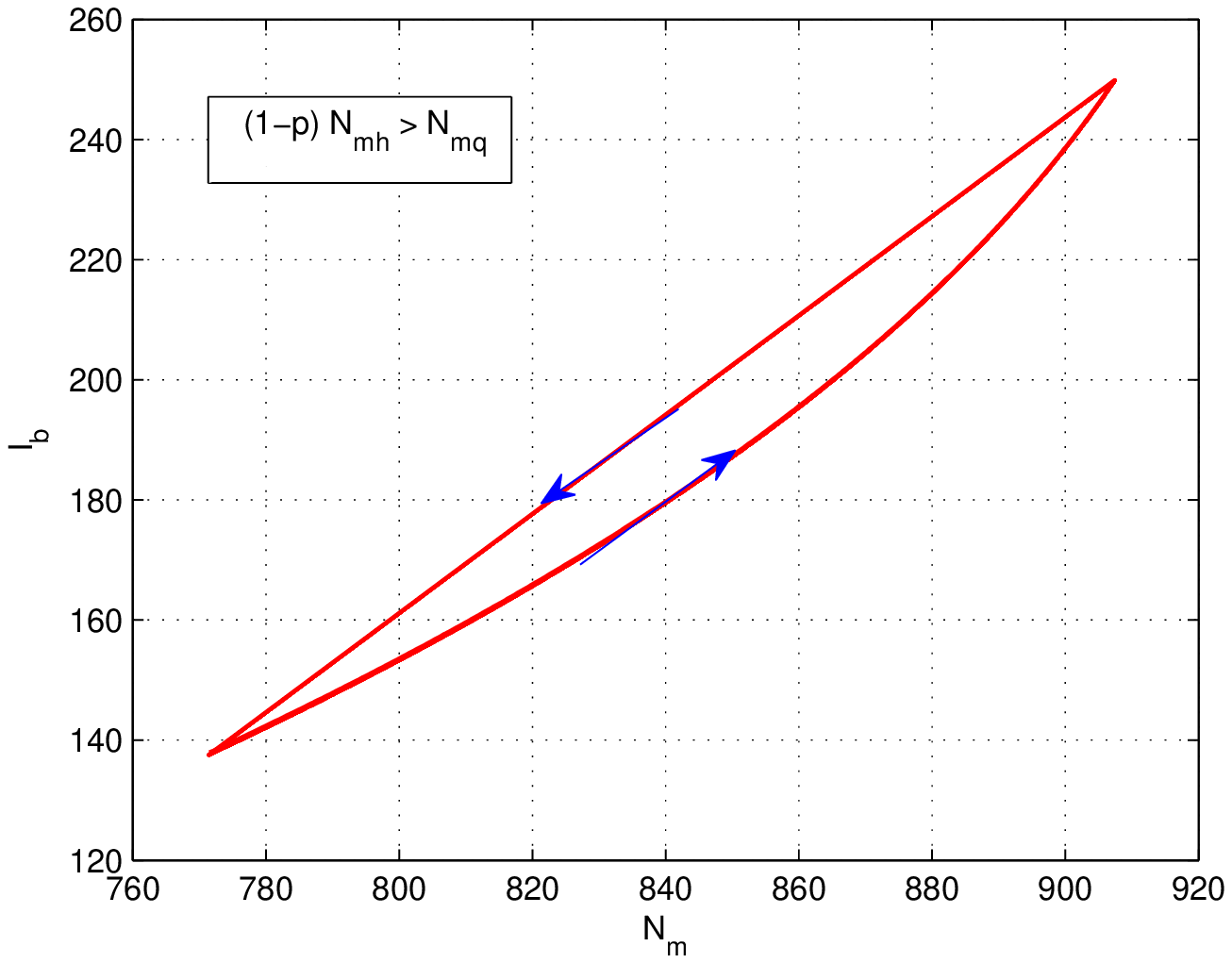}}
\subfigure[]{
\label{Fig.6(b)}
\includegraphics[width=0.46\textwidth]{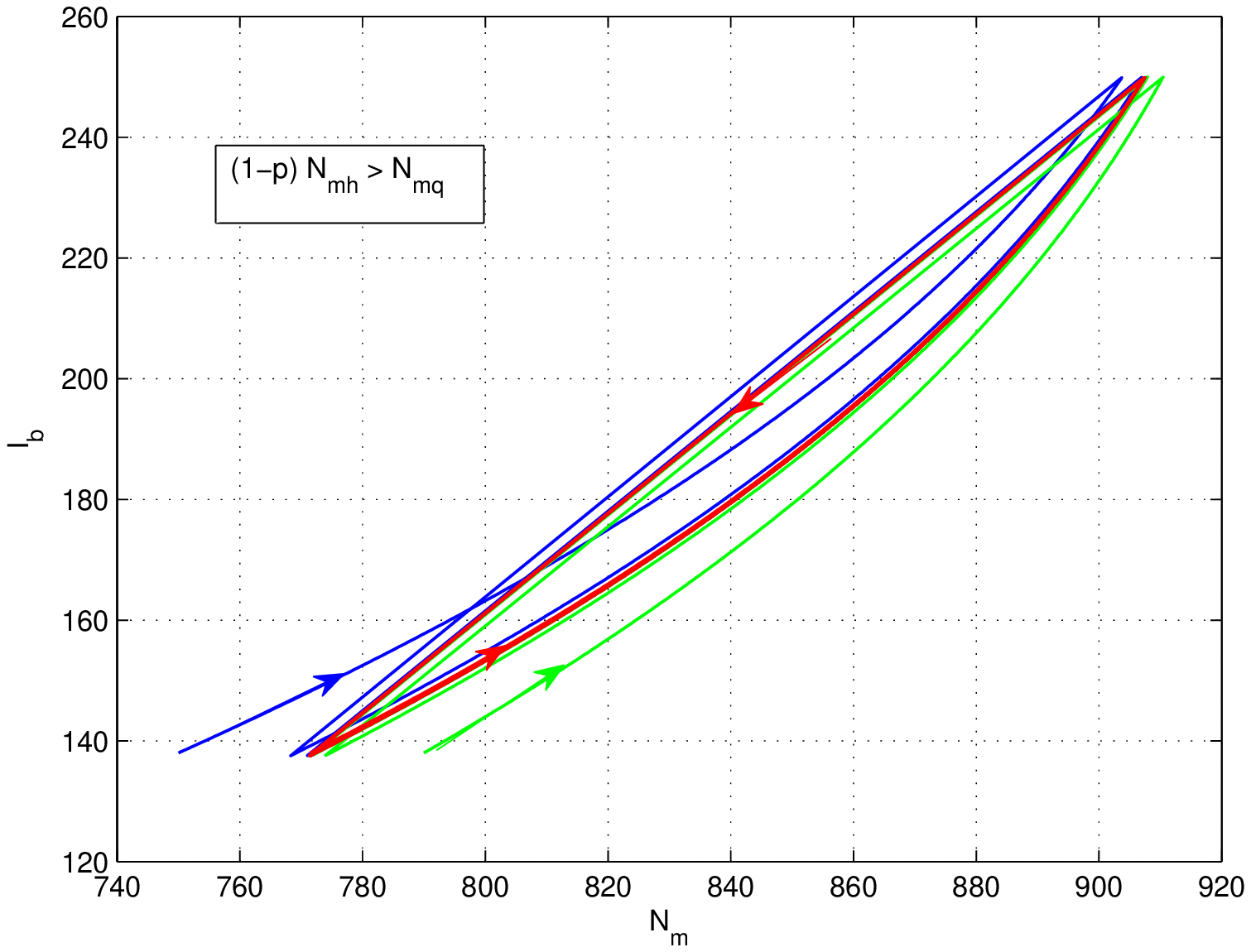}}
\caption{The existence of order-1 periodic solution of model \eqref{System-2}. (b)The orbitally asymptotical stability of order-1 periodic solution of model \eqref{System-2}. The baseline parameter values are as follows:  $\mu_m=0.357$, $\delta_m=0.035$, $K_m=1000$, $c=0.09$, $\beta_{bm}=0.8$, $N_b=400$, $\mu_b=0.01$, $p=0.15$, $q=0.45$, $H_b=250<400$ and the initial value $N_{m0}=771$, $I_{b0}=137$ in the case of $(1-p)N_{mh} > N_{mq}$.}
\label{Fig.6}
\end{figure}

Thirdly, the existence and orbitally asymptotical stability of order-1 periodic solution of model \eqref{System-2} is shown in Figure \ref{Fig.6(a)} and \ref{Fig.6(b)} in the case of $(1-p)N_{mh} > N_{mq}$. We know that the numerical simulation result is consistent with the therical reult of Theorem \eqref{Theorem-6}. Further, when the infection of West Nile Virus is not serious, we should take this strategies of state feedback control and satisfy this condition $(1-p)N_{mh} > N_{mq}$ so that the number of infected birds and mosquitoes should confine in a lower periodic range.

\begin{figure}[htbp]
\centering
\subfiguretopcaptrue
\subfigure[]{
\label{Fig.7(a)}
\includegraphics[width=0.46\textwidth]{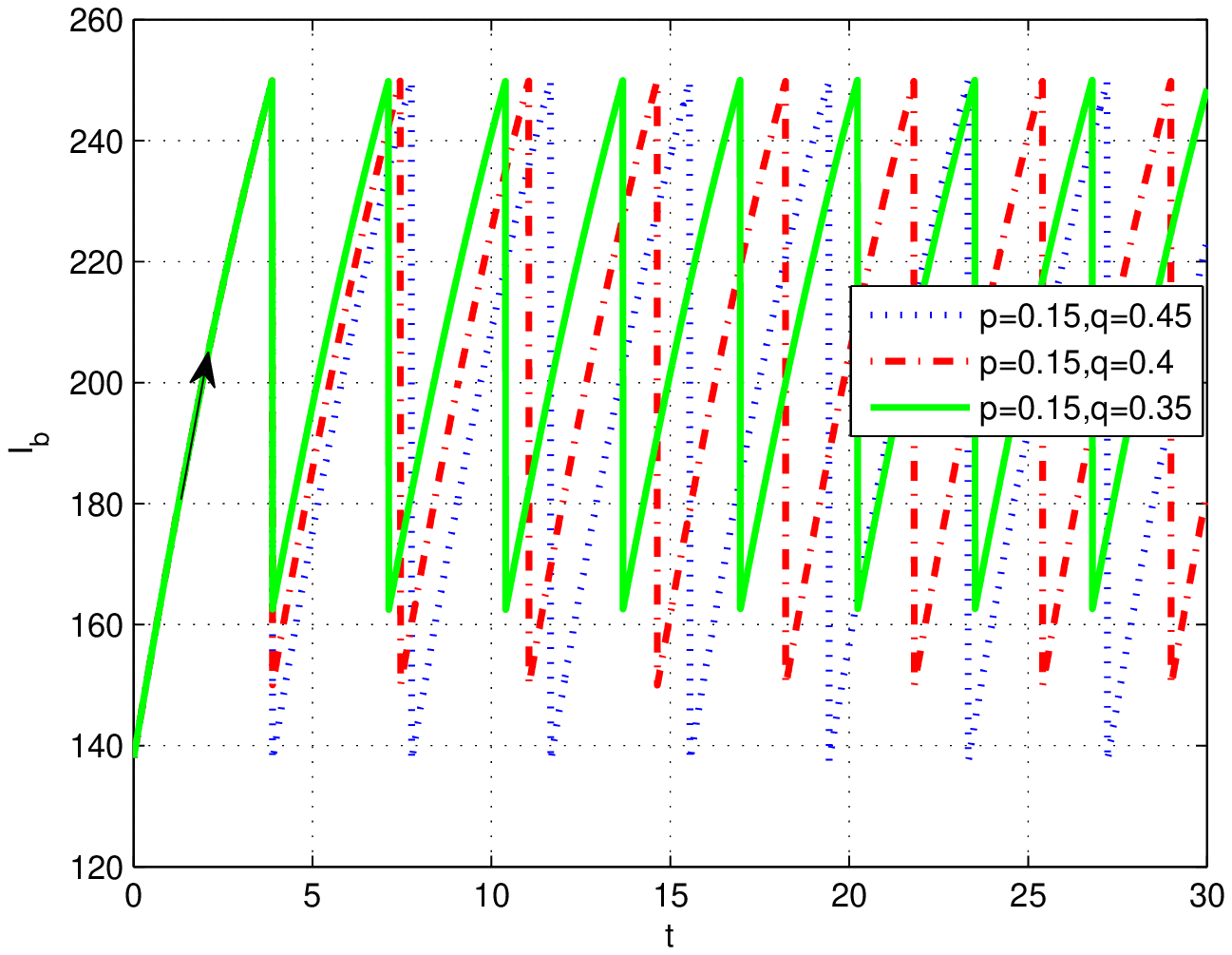}}
\subfigure[]{
\label{Fig.7(b)}
\includegraphics[width=0.46\textwidth]{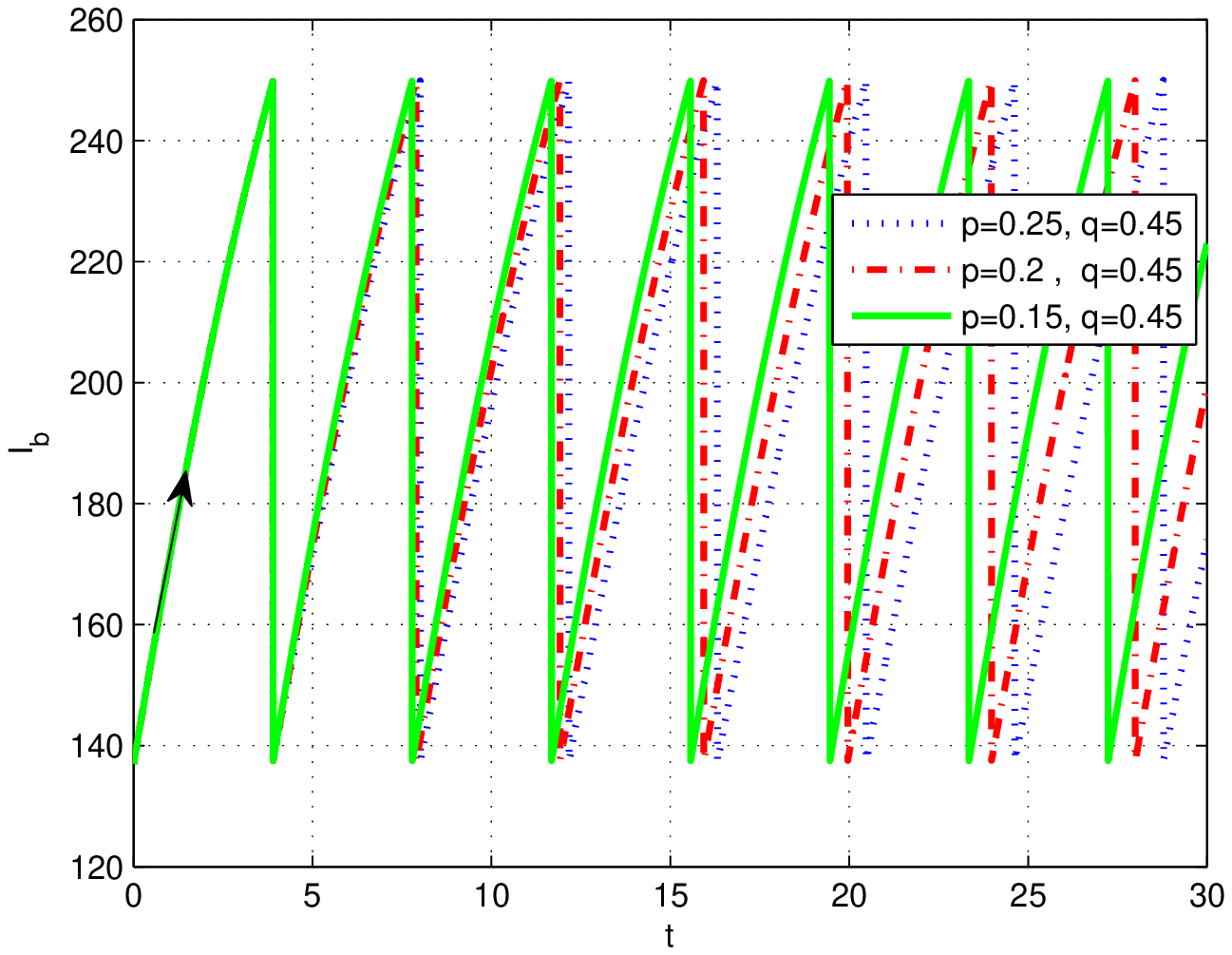}}
\caption{(a)The effects of $q$(the proportion of curing infected birds), where $q=0.45, 0.4, 0.35$ and $p=0.15$. (b)The effects of $p$(the proportion of culling mosquitoes), where $p=0.25, 0.2, 0.15$ and $q=0.45$. The other parameters are $\mu_m=0.537$, $\delta_m=0.035$, $K_m=1000$, $c=0.09$, $\beta_{bm}=0.8$, $N_b=400$, $\mu_b=0.01$, $H_b=250<400$,  and the initial value $N_{m0}=771$, $I_{b0}=137$ in the case of $(1-p)N_{mh} > N_{mq}$.}
\label{Fig.7}
\end{figure}

We find that the maximum amplitude of the number of infected birds $I_b$ dramatically decreased and the periodic of infected birds reach the threshold value $H_b=250$ obviously prolonged, when $q$ increased in Figure \ref{Fig.7(a)}. From Figure \ref{Fig.7(b)}, we draw that the maximum amplitude of the number of infected birds $I_b$ did not change, when $p$ increased. Further, we show that the periodic of infected birds extanded as the cofficient($q$) of curing infected birds increased. Therefore, the culling of mosquitoes $p$ and $q$ play an important role in controling the length of time in infected birds. Particularly, the cofficient($q$) of curing infected birds is more effective than the culling of mosquitoes $p$ in controlling infected birds with the $(1-p)N_{mh} > N_{mq}$.

\begin{figure}[htbp]
\centering
\subfiguretopcaptrue
\subfigure[]{
\label{Fig.8(a)}
\includegraphics[width=0.45\textwidth]{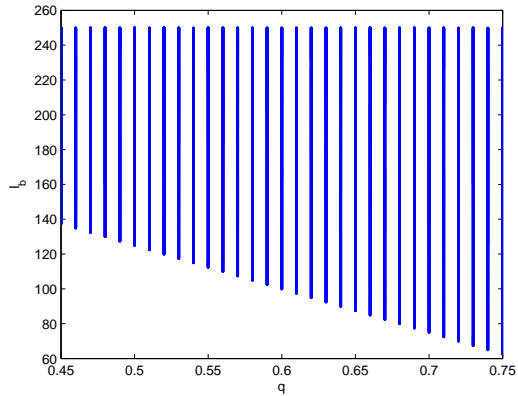}}
\caption{ (a) The bifurcation diagram of model \eqref{System-2} with respect to the parameter $q\in(0.45,\;0.75)$. The other parameters are $\mu_m=0.537$, $\delta_m=0.035$, $K_m=1000$, $c=0.09$, $\beta_{bm}=0.8$, $N_b=400$, $\mu_b=0.01$, $p=0.25$ and $H_b=250<400$.}
\label{Fig.8}
\end{figure}
The bifurcation diagram of model \eqref{System-2} with respect to the parameter $q\in(0.45,\;0.75)$ is shown in Figure \ref{Fig.8(a)}. The higher cured proportion of infected birds we choosed, the less infected birds are they.

\section{Concluding remarks}
In this paper, we built a new WNV mathematical model with impulsive state feedback
control. Firstly, we investigated qualitative characteristic of model \eqref{System-2} without impulsive effect, and obtained sufficient condition of globally asymptotically stable of the model \eqref{System-4}. Secondly, we decrive sufficient conditions for the existence and orbital stability of the positive order-1 periodic solution employing Poincar$\acute{e}$ map and the analogue of Poincar\'e criterion. Considering the threshold $H_b$ is how to have an effect on dynamic behavior of model \eqref{System-2}. There are two possible cases which are $(1-p)V_m^* < V_{mq}$ and $(1-p)V_{mh} > V_{mq}$. Existence and stability of periodic solution of order-1 or order-2 was yielded under the conditions $(1-p)V_m^* < V_{mq}$. The case, in which it was impossible that there was a periodic solution of order-$k(k\geq3)$ by differential equation geometry, theory of differential inequalities. Therefore, the culling of mosquitoes $p$ plays an important role in controling the length of time in infected birds. When $(1-p)V_{mh} > V_{mq}$ holds, we attained the existence and stability of periodic solution of order-1. The case, in which it was impossible that there was a periodic solution of order-$k(k\geq2)$ by theory of differential inequalities, the existence and uniqueness of the limit. Therefore, the culling of mosquitoes $p$ and $q$ play an important role in controling the length of time in infected birds. Particularly, the cofficient($q$) of curing infected birds is more effective than the culling of mosquitoes $p$ in controlling infected birds with the $(1-p)N_{mh} > N_{mq}$.

These results have important implications for curbing spread of the West Nile Virus, especially endangered birds.

\end{document}